\documentclass[10pt]{article}

\usepackage{amsmath,amssymb,amsthm,bm,comment,diagbox,enumerate,ifxetex,mathtools,microtype,scalerel,tikz,todonotes,booktabs,color}
\usepackage{lscape}
\usepackage[margin=1in]{geometry}

\usepackage{subcaption}

\ifxetex
  \usepackage{fontspec}
\else
  \usepackage[T1]{fontenc}
  \usepackage[utf8]{inputenc}
  \usepackage{lmodern}
\fi

\DeclareMathOperator{\Arith}{Arith}

\DeclareMathOperator{\SArith}{SArith}

\DeclarePairedDelimiter{\floor}{\lfloor}{\rfloor}

\newcommand{\overbar}[1]{\mkern 0.6mu\overline{\mkern-0.6mu#1\mkern-0.6mu}\mkern 0.6mu}

\newcommand{\vd}{\mathbf{d}}
\newcommand{\vr}{\mathbf{r}}
\newcommand{\vzero}{\mathbf{0}}

\renewcommand{\r}{\overbar{\mathbf{r}}}

\newcommand{\rr}{r}

\newcommand{\ZZ}{\mathbb{Z}}

\newcommand{\CP}{\mathcal{P}}
\newcommand{\p}{\mathfrak{p}_3}
\newcommand{\pp}{\mathfrak{p}_2}

\newtheorem{proposition}{Proposition}[section]
\newtheorem{conjecture}[proposition]{Conjecture}
\newtheorem{lemma}[proposition]{Lemma}
\newtheorem{corollary}[proposition]{Corollary}
\newtheorem{theorem}[proposition]{Theorem}
\newtheorem{definition}[proposition]{Definition}
\newtheorem*{genericthm*}{\thistheoremname}
\newenvironment{namedthm*}[1]
  {\renewcommand{\thistheoremname}{#1}%
   \begin{genericthm*}}
  {\end{genericthm*}}
\theoremstyle{definition}

\newcommand{\thistheoremname}{}

\title{Arithmetical Structures on Paths With a Doubled Edge}
\author{Darren Glass \& Joshua Wagner}
\date{}

\begin{document}
\maketitle

\begin{abstract}
An arithmetical structure on a graph is given by a labeling of the vertices which satisfies certain divisibility properties.  In this note, we look at several families of graphs and attempt to give counts on the number of arithmetical structures for graphs in these families.
\end{abstract}

\section{Introduction}

In this paper, we will consider the arithmetical structures on a particular family of graphs.  While arithmetical structures can be defined in several equivalent ways, we will take an approach that is based on elementary number theory.

\begin{definition}
An arithmetical structure on a graph $G$ is given by an assignment to each vertex of a nonnegative integer so that the set of all labels is relatively prime and so that the label of each vertex is a divisor of the sum of the label of its neighbors, counted with multiplicity where appropriate.
\end{definition}

This definition is equivalent to the definition given by Lorenzini in \cite{Lor}, in which he sets $A$ to be the adjacency matrix of a graph $G$ and defines an arithmetical structure to be a pair of vectors $(\vd, \vr) \in (\ZZ_{\ge 0})^n\times (\ZZ_{>0})^n$ so that the matrix $L(G,\vd)\coloneqq (\text{diag}(\vd)-A)$, satisfies the equation $L(G, \vd)\vr = \vzero$, with the additional restriction that the entries of $\vr$ are chosen to have no nontrivial common factor.

We denote the set of all arithmetical structures on a graph $G$ by $\Arith(G)$, and one question of interest is the size of this set.  Lorenzini proved in \cite{Lor} that the number of arithmetical structures on any given graph is finite, but his proof does not give a way to count or even bound the number of such structures.  Work of various authors in \cite{ICERM}, \cite{Oax}, \cite{CV2}, \cite{CV}, and more are able to count the number of structures for families of graphs such as paths, cycles, bidents, star graphs, and complete graphs.  In this paper, we consider the family of graphs consisting of a path where we double a single edge, as illustrated in Figure \ref{F:examples}.

\begin{definition}
The graph $\CP_{m,n}$ consists of vertices $\{a_1,\ldots,a_m,b_1,\ldots,b_n\}$ where there is a single edge between $a_i$ and $a_{i+1}$ for each $1 \le i < m$, a single edge between $b_i$ and $b_{i+1}$ for each $1 \le i < n$, and two edges between $a_1$ and $b_1$.
\end{definition}

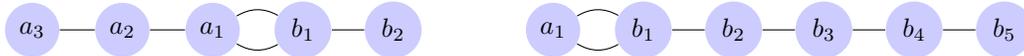
\begin{figure}[h]
\begin{center}
\begin{tikzpicture}
  [scale=.4,auto=left,every node/.style={circle,fill=blue!20}]
  \node (a) at (-6,0) {$a_3$};
  \node (b) at (-3,0) {$a_2$};
  \node (c) at (0,0) {$a_1$};
  \node (d) at (3,0) {$b_1$};
  \node (e) at (6,0) {$b_2$};
\foreach \from/\to in {a/b, b/c, d/e}
    \draw (\from) -- (\to);

\draw (c) to[bend left=30] (d);
\draw (c) to[bend right=30] (d);

\end{tikzpicture} \hskip .5in
\begin{tikzpicture}
  [scale=.4,auto=left,every node/.style={circle,fill=blue!20}]
  \node (c) at (0,0) {$a_1$};
  \node (d) at (3,0) {$b_1$};
  \node (e) at (6,0) {$b_2$};
  \node (f) at (9,0) {$b_3$};
  \node (g) at (12,0) {$b_4$};
  \node (h) at (15,0) {$b_5$};
\foreach \from/\to in {d/e, e/f, f/g, g/h}
    \draw (\from) -- (\to);

\draw (c) to[bend left=30] (d);
\draw (c) to[bend right=30] (d);

\end{tikzpicture}
\end{center}
\caption{The graphs $\CP_{3,2}$ and $\CP_{1,5}$}
\label{F:examples}
\end{figure}

We note that for ease of notation we will refer interchangeably to the vertex $a_i$ and to the numerical label assigned to it in a given arithmetical structure.  We should also note that when considering the doubled edge one must count the neighbors with multiplicity; in particular, the condition for an arithmetical structure includes that $a_1|(a_2+2b_1)$ and $b_1|(b_2+2a_1)$.  Examples of arithmetical structures on these graphs are given in Figure \ref{F:examples2}.

\begin{figure}[!h]
\begin{center}
\begin{tikzpicture}
  [scale=.4,auto=left,every node/.style={circle,fill=blue!20}]
  \node (a) at (-6,0) {$2$};
  \node (b) at (-3,0) {$4$};
  \node (c) at (0,0) {$6$};
  \node (d) at (3,0) {$13$};
  \node (e) at (6,0) {$1$};
\foreach \from/\to in {a/b, b/c, d/e}
    \draw (\from) -- (\to);

\draw (c) to[bend left=30] (d);
\draw (c) to[bend right=30] (d);

\end{tikzpicture} \hskip .5in
\begin{tikzpicture}
  [scale=.4,auto=left,every node/.style={circle,fill=blue!20}]
  \node (c) at (0,0) {$1$};
  \node (d) at (3,0) {$5$};
  \node (e) at (6,0) {$8$};
  \node (f) at (9,0) {$3$};
  \node (g) at (12,0) {$1$};
  \node (h) at (15,0) {$1$};
\foreach \from/\to in {d/e, e/f, f/g, g/h}
    \draw (\from) -- (\to);

\draw (c) to[bend left=30] (d);
\draw (c) to[bend right=30] (d);

\end{tikzpicture}
\end{center}
\caption{An example of a smooth structure on $\CP_{3,2}$ and a non-smooth structure on $\CP_{1,5}$.}
\label{F:examples2}
\end{figure}
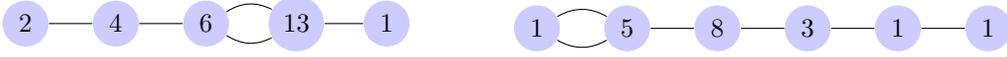

In order to help us understand the arithmetical structures on this family, we first define a subset of those structures which we will call {\it smooth.}  We will denote the set of smooth arithmetical structures by $\SArith(\CP_{m,n})$.

\begin{definition}
An arithmetical structure on $\CP_{m,n}$ is said to be smooth if we have that $a_1 > a_2 > \ldots > a_m$ and $b_1 > b_2 > \ldots > b_n$.
\end{definition}

We note that in the language of $\vd$-vectors from Lorenzini's definition, this corresponds to the fact that all vertices other than possibly $a_1$ and $b_1$ have associated $\vd$-values greater than one. Moreover, in a smooth arithmetical structure we have that $0<a_i<a_{i-1}$ and $a_{i-1}|a_i+a_{i-2}$, from which it follows that $a_i$ must be the least residue of $-a_{i-2}$ mod $a_{i-1}$.  In this way, we see that knowing $a_1$ and $a_2$ will determine the values of all $a_i$ in a smooth arithmetical structure and the analogous results hold for the $b_i$.

One goal of this note is to count the number of smooth arithmetical structures on $\CP_{1,n}, \CP_{2,n}, \CP_{3,n}$.  In particular, we will see in Theorems \ref{T:smoothP1n}, \ref{T:smoothP2n}, and \ref{T:smoothP3n} that in each of these cases the number of smooth structures on $\CP_{m,n}$ grows on the order of a polynomial of degree $m-1$ in $n$.  This suggests the following conjecture to the authors, which we hope to explore going forward:

\begin{conjecture}
The number of smooth arithmetical structures on the graph $\CP_{m,n}$ grows at the same rate as $\binom{m+n-1}{n}$.
\end{conjecture}

We are interested in smooth structures because every arithmetical structure can be reduced to a smooth structure by a pair of operations that we jointly refer to as {\it smoothing} and which we define in Section \ref{S:background}.  If a structure $(\vr,\vd)$ on a graph $G$ is obtained by a sequence of smoothing operations on structure $(\vr',\vd')$ on a graph $H$, we say that $(G,(\vr,\vd))$ is an ancestor of $(H,(\vr',\vd'))$.  Corollary \ref{C:count} will give us a way of counting the number of structures that are derived from the same ancestor, and this will allow us to derive a formula for the total number of arithmetical structures from the total number of smooth structures.  Sections \ref{S:P1n}, \ref{S:P2n}, and \ref{S:P3n} consider these questions for the cases of $m=1,2,3$ respectively. One consequence of our results is the following:

\begin{theorem}
For large $n$, there are approximately $\frac{7}{2}C_n$ arithmetical structures on $\CP_{1,n}$, approximately $\frac{76523}{57600} C_{n}$ structures on $\CP_{2,n}$, and approximately $\frac{78157}{600} C_n$ structures on $\CP_{3,n}$, where $C_n = \frac{1}{n+1}\binom{2n}{n}$ is the $n^{th}$ Catalan number.
\end{theorem}


We conclude this introduction by noting that often when one considers arithmetical structures one is also interested in the Critical Group associated to the structure, which can be defined explicitly as the cokernel of the map from $\ZZ^{m+n} \rightarrow \ZZ^{m+n}$ defined by $x \mapsto L(G,\vd)x$. It follows as an immediate consequence of \cite[Cor 2.3]{Lor} that the order of the critical group associated to an arithmetical structure on $\CP_{m,n}$ is given by $\frac{2}{a_mb_n}$.  In particular, this shows that for any arithmetical structure we either have $a_m=b_n=1$, in which case the associated critical group is $\ZZ/2\ZZ$, or exactly one of these numbers is equal to $2$, in which case the critical group is trivial.  We will prove this condition on $a_m$ and $b_n$ directly in Theorem \ref{T:endpts}, but we note that Lorenzini's result allows us to explicitly determine the critical group for all structures.


\section{Background on Smooth Structures}\label{S:background}

We begin this section by defining two smoothing operations on arithmetical structures.  While we will define these operations on the graphs $\CP_{m,n}$ the definitions are more general and are essentially the same as definitions given in \cite{ICERM} and \cite{Oax}.  In particular, they are equivalent to removing vertices of degree one or two where the $\vd$-vector has value $1$.

\begin{definition}
The following two operations, along with the analogous operations on the $a_i$, are jointly referred to as smoothing a vertex:
\begin{itemize}
\item If we are given an arithmetical structure on a graph $\CP_{m,n}$ so that $b_n=b_{n-1}$ then we obtain a new arithmetical structure on $\CP_{m,n-1}$ by removing the vertex $b_n$ and leaving the other values of $a_i,b_i$ the same.
\item If we are given an arithmetical structure on a graph $\CP_{m,n}$ so that $b_i=b_{i-1}+b_{i+1}$ then we obtain a new arithmetical structure on $\CP_{m,n-1}$ by setting $a_j'=a_j$ for all $j$, $b_j'=b_j$ for $j<i$ and $b_j'=b_{j+1}$ for $j \ge i$.
\end{itemize}
If there are no vertices satisfying either of these cases then we refer to the structure as smooth.  We will say that a structure on $\CP_{m',n'}$ is an ancestor of a structure on $\CP_{m,n}$ if it is obtained from a series of smoothing operations.
\end{definition}

\begin{figure}
\begin{center}
\begin{tikzpicture}
  [scale=.4,auto=left,every node/.style={circle,fill=blue!20}]
  \node (c) at (0,0) {$1$};
  \node (d) at (3,0) {$5$};
  \node (e) at (6,0) {$8$};
  \node (f) at (9,0) {$3$};
  \node (g) at (12,0) {$1$};
  \node (h) at (15,0) {$1$};
\foreach \from/\to in {d/e, e/f, f/g, g/h}
    \draw (\from) -- (\to);

\draw (c) to[bend left=30] (d);
\draw (c) to[bend right=30] (d);

\end{tikzpicture}

\begin{tikzpicture}
  [scale=.4,auto=left,every node/.style={circle,fill=blue!20}]
  \node (c) at (0,0) {$1$};
  \node (d) at (3,0) {$5$};
  \node (e) at (6,0) {$8$};
  \node (f) at (9,0) {$3$};
  \node (g) at (12,0) {$1$};
\foreach \from/\to in {d/e, e/f, f/g}
    \draw (\from) -- (\to);

\draw (c) to[bend left=30] (d);
\draw (c) to[bend right=30] (d);

\end{tikzpicture}

\begin{tikzpicture}
  [scale=.4,auto=left,every node/.style={circle,fill=blue!20}]
  \node (c) at (0,0) {$1$};
  \node (d) at (3,0) {$5$};
  \node (f) at (6,0) {$3$};
  \node (g) at (9,0) {$1$};
\foreach \from/\to in {d/f, f/g}
    \draw (\from) -- (\to);

\draw (c) to[bend left=30] (d);
\draw (c) to[bend right=30] (d);

\end{tikzpicture}
\end{center}
\label{F:Ex3}
\caption{A structure on $\CP_{1,5}$ and the results of two smoothing operations leading to the unique smooth ancestor}
\end{figure}
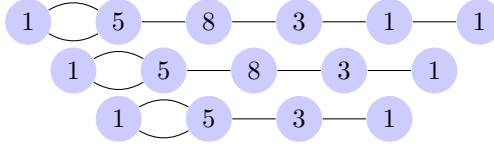

Given an arithmetical structure it is clear that one can obtain a smooth arithmetical structure after some sequence of smoothing operations.  Moreover, one can see that the smooth structure one gets will consist precisely of the maximal decreasing subsequences of the $\{a_i\}$ and $\{b_i\}$, and in particular each structure will have a unique smooth ancestor. In order to count the number of arithmetical structures on $\CP_{m,n}$ that have the same smooth ancestor, we will first define a function $C(n,k)=\frac{k}{n} \binom{2n-k-1}{n-1}$.  We note that this is equal to $B(n-1,n-k)$, where  $B(s,t)=\frac{s-t+1}{s+1}\binom{s+t}{s}$ denotes the ballot numbers, a generalization of the Catalan numbers that were first studied by Carlitz \cite{Carlitz} and are defined for all $s \ge t \ge 0$.  The ballot numbers and Catalan numbers are ubiquitous in combinatorics and have many interpretations.  The proofs of \cite[Theorem 9]{Oax} and \cite[Lemma 2.9]{ICERM} can very easily be adapted to show that $B(s,t)$ counts the number of arithmetical structures on a path of length $s$ that are descendents of a given structure on a path of length $s-t$.  From this, we obtain the following result:

\begin{theorem} \label{T:ancestors}
Any arithmetical structure on $\CP_{m,n}$ has a unique ancestor which is a smooth arithmetical structure on some $\CP_{m',n'}$.  Moreover, for each smooth arithmetical structure on $\CP_{m',n'}$ there are $C(m,m')C(n,n')$ arithmetical structures on $\CP_{m,n}$ having it as an ancestor.
\end{theorem}

The following corollary, which is an immediate consequence, allows us to count the number of arithmetical structures on a given graph $\CP_{m,n}$ from the number of smooth structures on each smaller $\CP_{m',n'}$.  We note that when we define these sets we are specifying the two `sides' of the path, so for example on the graph $\CP_{1,1}$ we are counting the structure where $a_1=2, b_1=1$ as different from the structure where $a_1=1,b_1=2$.  When $m=n$ it might be more natural to divide all of the counts by two to account for the inherent symmetry in the graphs.

\begin{corollary}\label{C:count}
The number of arithmetical structures on $\CP_{m,n}$ can be derived from the number of smooth structures by the following formula:
\[|\Arith(\CP_{m,n})|=\sum_{m'=1}^m \sum_{n'=1}^n (C(m,m')C(n,n')) |\SArith(\CP_{m',n'})|\]
\end{corollary}

To use this corollary throughout the paper, we will find the following results on the function $C(n,k)$ useful to have so we state them here.

\begin{lemma}\label{L:C}
The function $C(n,k)=\frac{k}{n} \binom{2n-k-1}{n-1}$ satisfies the following properties:
\begin{enumerate}[(a)]
\item $C(n,n)=1$.
\item $C(n,n-1) = n-1$.
\item $C(n,1) = C(n,2) = \frac{1}{n}\binom{2n-2}{n-1} = C_{n-1}$, the Catalan number.
\item $C(n,3) =C_{n-1}-C_{n-2}$ and $C(n,4)=C_{n-1}-2C_{n-2}$.
\item $\sum_{k=s}^n C(n,k) = C(n+1,s+1)$, and in particular $\sum_{k=1}^n C(n,k) = C_n$.
\item $\sum_{k=1}^n kC(n,k) = C_{n+1}-C_n$.
\item $\sum_{k=1}^n k^2C(n,k) = 2C_{n+2}-5C_{n+1}+C_n$
\item $\lim_{n \to \infty} \frac{C(n,k)}{C_n} = \frac{k}{2^{k+1}}$.
\end{enumerate}
\end{lemma}

\begin{proof}
The proof of the first four parts of this lemma are immediate from the definition.

To prove the next three parts, we will use a standard interpretation of the ballot numbers and Catalan numbers based on lattice paths.  In particular, we note that $C(n,k)$ gives the number of paths from the point $(0,0)$ to the point $(n-k,n-1)$ that travel only along grid lines and avoid all points $(x,y)$ with $x>y$.  Additionally, we have $C_n$ as the number of such paths from $(0,0)$ to $(n,n)$.  See \cite{Stanley} for more details on this interpretation.

In particular, $C(n+1,s+1)$ will count the number of lattice paths to the point $(n-s,n)$.  Each such path will traverse exactly one edge from a point $(j,n-1)$ to $(j,n)$ where $0 \le j \le n-s$, and in particular the number of paths that cross each such edge for a fixed $j$ is exactly the number of lattice paths to the point $(j,n-1)$, which is given by $C(n,n-j)$.  This tells us that $\displaystyle \sum_{j=0}^{n-s}C(n,n-j)=C(n+1,s+1)$, and statement (e) follows from the a change of variables.

In a similar vein, we note that $C_{n+1}-C_n$ counts the number of paths from $(0,0)$ to $(n+1,n+1)$ which avoid the point $(n,n)$ and therefore must include the point $(n-1,n+1)$. Now, any such path must include exactly one edge from a point $(j,n-1)$ to $(j,n)$ where $0 \le j \le n-1$.  Fixing one such $j$ we note that the number of paths including that edge is equal to the product of the number of paths from $(0,0)$ to $(j,n-1)$ and the number of paths from $(j,n)$ to $(n-1,n+1)$.  There are $C(n,n-j)$ of the former and $n-j$ of the latter, so we have that $\sum_{j=0}^{n-1} (n-j)C(n,n-j) = C_{n+1}-C_n$.  Part (f) follows by setting $k=n-j$.

The proof of part (g) works in a similar manner, conditioning the paths to $(n+2,n+2)$ on the edge they contain from $(j,n-1)$ to $(j,n)$ in the same way as the previous paragraph in order to obtain the identity $\displaystyle \sum_{k=1}^n \frac{1}{2} (k^2+5k+4)C(n,k)=C_{n+2}$, from which the claim follows.

The proof of the final part is a straightforward exercise in computing limits, after noting that
\[\frac{C(n,k)}{C_n}=\frac{k(2n-k-1)!(n+1)!}{(n-k)!(2n)!}\].

\end{proof}

In order to understand the set $\SArith(\CP_{m,n})$, we will first show that a smooth arithmetical structure is uniquely determined by the choice of $a_1$ and $b_1$.  In order to do this, we define an auxiliary function $F(x_1,x_2)$.

\begin{definition}
Let $x_1 >0$ and $x_2 \ge 0$.  We define a sequence $\{x_i\}$ by letting $x_i$ be the least residue of $-x_{i-2} \pmod x_{i-1}$. We then define the function $F(x_1,x_2)$ to be the largest $i$ so that $x_i >0$.
\end{definition}

For example, if $x_1=8$ and $x_2=5$ we compute that $x_3$ is the least residue of $-8 \pmod 5$, which is $2$.  We then get that $x_4 = 1, x_5=0$.  In particular, $F(8,5)=4$.

\begin{lemma}
For any pair of relatively prime integers $(a_1,b_1)$ there is at most one smooth arithmetical structure on a graph of the form $\CP_{m,n}$.  Moreover, we have that $m=F(2b_1,a_1)-1$ and $n=F(2a_1,b_1)-1$.
\end{lemma}

\begin{proof}
Given the pair $(a_1,b_1)$ we define $b_2$ to be the least residue of $-2a_1 \pmod b_1$, so that $b_1|2a_1+b_2$.  For $i>2$ we define $b_i$ to be the least residue of $-b_{i-2}$ mod $b_{i-1}$.  We define $a_i$ in an analogous way for $i>1$.  As long as $b_i>0$ we therefore have that $b_i<b_{i-1}$, and the same is true for the $a_i$, so the nonzero entries of this sequence satisfy the divisibility requirements of a smooth arithmetical structure on $\CP_{F(2b_1,a_1)-1,F(2a_1,b_1)-1}$.  The fact that $a_1$ and $b_1$ are relatively prime implies that the $gcd$ of the entire set is also one.
\end{proof}

One approach to counting smooth structures on $\CP_{m,n}$ would therefore be to try to understand the number of pairs $(a_1,b_1)$ so that $F(2b_1,a_1)=m+1$ and $F(2a_1,b_1)=n+1$. Unfortunately, we can not write down a concise formula for $F$, let alone invert it, so this direct approach seems out of reach.  On the other hand, there are a number of properties of the function $F(x,y)$ that are discussed at length in \cite{ICERM} that will be useful.  We will omit proofs of the following two of these facts, although they are not difficult.

\begin{lemma}\label{L:props}
For any $x >0$ and $y \ge 0$ we have the following:
\begin{itemize}
\item $F(x+ky,y)=F(x,y)$
\item $F(x,kx+y) = k+F(x,y)$
\end{itemize}
\end{lemma}

One consequence of this lemma that will come in handy later is the following:

\begin{lemma}\label{L:4}
For any $x$, we have $F(4,x) = \frac{x+\epsilon_x}{4}$ where $\epsilon_x= \begin{cases}
4, & x \equiv 0 \pmod 4 \\
7, & x \equiv 1 \pmod 4 \\
6, & x \equiv 2 \pmod 4 \\
13, & x \equiv 3 \pmod 4 \\
\end{cases}$
\end{lemma}

\begin{proof}
Write $x=4k+\epsilon$ where $0 \le \epsilon < 4$.  Then it follows from Lemma \ref{L:props} that we have $F(4,x) = k + F(4,\epsilon)$.  If $\epsilon=0$ we note that $F(4,\epsilon)=1$ so $F(4,x)=k+1=\frac{x+4}{4}$.  If $\epsilon=1$ (resp. $2$) then $F(4,\epsilon)=2$ so we have $F(4,x)=k+2$, which is equal to $\frac{x+7}{4}$ (resp. $\frac{x+6}{4}$).  Finally, if $\epsilon=3$ we have that $F(4,\epsilon)=4$ so $F(4,x)=\frac{x-3}{4}+4=\frac{x+13}{4}$.
\end{proof}

Note that in the introduction we defined an arithmetical structure to include the criterion $gcd(\{a_i,b_i\})=1$.  This is simply a way of specifying a single representative in the equivalence class of possible $\vr$-vectors.  The following theorem shows in an elementary manner that in such a representation we have that $a_m=1$ or $2$.  For this reason, each arithmetical structure also has a representative in which $a_m=2$ which can be obtained by scaling the entire structure by a factor of $2$ if needed.  We will refer to this as a $\r$-vector for the structure, and in subsequent sections it often turns out to be easier to count these.

\begin{theorem}\label{T:endpts}
If we have an arithmetical structure on $\CP_{m,n}$ represented by its $\vr$-vector then $a_m \in \{1,2\}$
\end{theorem}

\begin{proof}
Note that we have that $a_m|a_{m-1}$.  Moreover, we have that $a_{m-1}|(a_m+a_{m-2})$, which implies in turn that $a_m|a_{m-2}$.  We similarly see that $a_m|a_i$ for all $i$.  The fact that $a_1|(a_2+2b_1)$ and $a_m$ divides both $a_1$ and $a_2$ now shows that $a_m|2b_1$, and similar to the above we see that $a_m|2b_i$ for all $i$. By symmetry, we also have that $b_n|2a_i$ for all $i$.  The fact that the set $\{a_i,b_i\}$ has greatest common divisor equal to one now implies that $gcd(a_m,b_n)=1$, but if $a_m|2b_n$ then we must have that $a_m=1$ or $2$.
\end{proof}

\section{Graphs of the form $\CP_{1,n}$}\label{S:P1n}

In this section, we wish to count the number of arithmetical structures on the graph $\CP_{1,n}$.  It follows from Corollary \ref{C:count} that our first step should be to count the number of smooth structures.

%

Given a smooth arithmetical structure, we will consider the $\r$-vector discussed in the previous section.  In particular, we will set $a_1=2$ and fix $n$ and ask how many choices of $b_1$ there are so that $F(2a_1,b_1)=n+1$  and $F(2b_1,a_1)=2$. The latter condition is immediate as $a_1|2b_1$, and it follows from Lemma \ref{L:4} that we have \[F(2a_1,b_1)=F(4,b_1)=\frac{b_1+\epsilon_{b_1}}{4}\] In particular, if $b_1 \equiv 1 \pmod 4$ then we want to set $n+1 = \frac{b_1+7}{4}$, so that $b_1=4n-3$.  Considering the other cases of $b_1 \pmod 4$ we see that for each $n$ we will get a structure on $\CP_{1,n}$ by choosing $b_1=4n$ or $4n-2$, and if $n \ge 3$ we will get an additional structure by setting $b_1=4n-9$.

In particular, we have proven the following:

\begin{theorem}\label{T:smoothP1n}
There are exactly four smooth arithmetical structures on $\CP_{1,n}$ for all $n \ge 3$.  There are $3$ smooth arithmetical structures on each of $\CP_{1,1}$ and $\CP_{1,2}$.
\end{theorem}

This allows us to prove the following count on the total number of arithmetical structures.

\begin{theorem}\label{T:P1n}
The number of arithmetical structures on $\CP_{1,n}$ is given by $4C_n-2C_{n-1}$
\end{theorem}

\begin{proof}
We have from Corollary \ref{C:count} that the number of arithmetical structures on $\CP_{1,n}$ is equal to
\begin{eqnarray*}
\sum_{n'=1}^n C(n,n')|\SArith(\CP_{1,n'})| &=& 3C(n,1)+3C(n,2)+\sum_{n'=3}^n 4C(n,n') \\
&=& \sum_{n'=1}^n 4C(n,n') - C(n,1)-C(n,2) \\
&=&4C_n-2C_{n-1}
\end{eqnarray*}
\end{proof}

It is well known that for large $n$ we have $C_n \approx \frac{4^n}{\sqrt{\pi}n^{3/2}}$, so in particular we see that for large $n$ we have that the number of arithmetical structures on $\CP_{1,n}$ is approximately $\frac{7}{2}C_n \approx \frac{7 \cdot 4^n}{2\sqrt{\pi}n^{3/2}}$.

\section{Graphs of the form $\CP_{2,n}$}\label{S:P2n}

We will begin by counting the number of smooth arithmetical structures on the graph $\CP_{2,n}$.  We will assume for simplicity that $n>1$, as we have considered the $n=1$ case in the previous section.  We will again count the number of possible $\r$-structures.  Recall in particular that this means that $a_2=2$ and $a_1=2a$ for some $a>1$.  Setting $b_1=b$, we are trying to find pairs $(a,b)$ so that $2a | 2b+2$ and $F(4a,b)=n+1$; any such pair satisfying these conditions will give a unique smooth structure. Our goal will be to compute the number of such pairs $(a,b)$ satisfying these conditions.

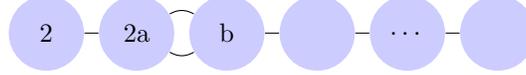
\begin{center}
\begin{tikzpicture}
  [scale=.4,auto=left,every node/.style={circle,fill=blue!20,minimum size=1cm}]
  \node (b) at (-3,0) {2};
  \node (c) at (0,0) {2a};
  \node (d) at (3,0) {b};
  \node (e) at (6,0) {};
  \node (f) at (9,0) {\ldots};
  \node (g) at (12,0) {};

\foreach \from/\to in {b/c, d/e,e/f,f/g}
    \draw (\from) -- (\to);

\draw (c) to[bend left=30] (d);
\draw (c) to[bend right=30] (d);

\end{tikzpicture}
\end{center}

The fact that $2a|2b+2$ implies that $b \equiv -1 \pmod a$, so we let $b=ka-1$.  We now note that we wish for $n+1=F(4a,ka-1)$.  The right hand side of this equation will simplify in different ways depending on the value of $k \pmod 4$.

\begin{itemize}
\item  If $k \equiv 1 \pmod 4$ then we have from Lemma \ref{L:props} that \[F(4a,b) = F(4a,ka-1)= \frac{k-1}{4} + F(4a,a-1) = \frac{k-1}{4}+ F(4,a-1).\]  We can use Lemma \ref{L:4} to further simplify to get $F(4a,b) =\frac{k-1}{4} + \frac{a-1+\epsilon_{a-1}}{4}$.  Knowing that this value should be equal to $n+1$ allows us to compute that $k=4n+6-a-\epsilon_{a-1}$, or in other words that $b=-a^2+(4n+6-\epsilon_{a-1})a-1$.  We want $b > 0$, so any choice of $a$ that makes this true will give us a smooth structure.  In particular for each $\epsilon \in \{0,1,2,3\}$ we wish to count the number of choices of $a$ that are congruent to $\epsilon \pmod 4$ that make $a^2-(4n+6-\epsilon_{a-1}) + 1 <0$. To see which values of $a$ make this true, we will make use of the following result:

\begin{lemma}\label{L:roots}
Let $\gamma>2$ be an integer.  Then the function $f(x)=x^2-\gamma x +1$ will be negative for all integers $x \in [1, \gamma-1]$ and positive for all other integers.
\end{lemma}

\begin{proof}
It is a simple algebra exercise to see that $f(x)<0$ for all values of $x$ in the range $\left( \frac{\gamma-\sqrt{\gamma^2-4}}{2},\frac{\gamma+\sqrt{\gamma^2-4}}{2}\right)$.  One can show using calculus that the left endpoint is in the range $(0,1)$ and the right endpoint is in the range $(\gamma-1,\gamma)$. In fact, one can see directly that $f(0)=f(\gamma)=1>0$ and $f(1)=f(\gamma-1) = 2-\gamma<0$. Thus, the two roots of $f(x)$ are in the ranges $(0,1)$ and $(\gamma-1,\gamma)$ so the result follows from properties of quadratic equations.
\end{proof}

In particular, if $\epsilon=0$ we are counting the number of values of $a$ that are congruent to $0 \pmod 4$ in the range $[1,4n+6-\epsilon_3-1] = [2,4n-8]$, so (because $n \ge 2$) we have that there are $n-2$ choices of $a$ that work.  Similarly, if $\epsilon=1$ we are counting the number of choices of $a \equiv 1$ mod $4$ in the range $[2,4n+1]$, of which there are $n$.  The conditions are also satisfied if $a \equiv 2 \pmod 4$ and $a \in [1,4n-2]$ or if $a \equiv 3 \pmod 4$ and $a \in [1,4n-1]$.  Summing these up gives us a total of $4n-2$ structures on $\CP_{2,n}$.

\item If $k \equiv 2 \pmod 4$, we can compute that \[F(4a,b)=F(4a,ka-1)=F(4a,2a-1)+\frac{k-2}{4} = F(2,2a-1)+\frac{k-2}{4} = a+1+\frac{k-2}{4}.\]  In particular, we have that $n=\frac{k-2}{4}+a$, from which we can conclude that $k=4n-4a+2$ so that $b=-(2a)^2+(2n+1)2a -1$.  It follows from Lemma \ref{L:roots} that we will have a structure for each $a>1$ so that $2a \in [1,2n]$.  In particular, there are $n-1$ smooth structures that can be formed in this way.

\item If $k \equiv 3 \pmod 4$ then we can compute that \[F(4a,b) = F(a+1,3a-1) + \frac{k-3}{4} = F(a+1,a-3)+\frac{k+5}{4} = F(4,a-3)+\frac{k+5}{4}\] This is equal to $\frac{a+k+2+\epsilon_{a-3}}{4}$.  We note that this computation only works if $a \ge 3$, but separate direct computations will show that if $a=2$ then we get a structure whenever $n \ge 3$. As before, setting $F(4a,b)=n+1$ allows us to compute that $b=-a^2+(4n+2-\epsilon_{a-3})-1$.  A computation similar to the above shows that this will be positive if $a \in [3,4n-11] \cup \{4n-9,4n-8,4n-7,4n-5\}$. Therefore, in this case we have $4n-8$ structures for each $n > 3$, $5$ structures for $n=3$, and no structures when $n=2$.

\item Finally, if $4|k$ then we compute that \[n=F(4a,b)-1=F(4a,ka-1)-1=\frac{k-4}{4}+F(4a,4a-1)-1 = \frac{k}{4}+4a-2.\]  Solving for $k$ implies that $k=4n-16a+8$ and in particular that $b=-16a^2+(4n+8)a-1$.  Lemma \ref{L:roots} implies that $b$ will therefore be negative if $4a \in [1, n+2)$, which will happen precisely if $1 < a < \frac{n+2}{4}$.  There are $\floor*{\frac{n-3}{4}}$ such choices of $a$.
\end{itemize}
Combining these four cases gives us the following theorem:

\begin{theorem}\label{T:smoothP2n}
If $n \ge 4$ then there are $9n-11 + \floor*{\frac{n-3}{4}}$ smooth arithmetical structures on the graph $\CP_{2,n}$.  Moreover, there are $17$ smooth structures on $\CP_{2,3}$, eight smooth structures on $\CP_{2,2}$ and three smooth structures on $\CP_{2,1}$.
\end{theorem}

More specifically, when $n \ge 4$ one notes that one gets four arithmetical structures for each value of $a \in \left[2,\floor*{\frac{n-3}{4}}\right]$, three structures for each $a \in \left[\floor*{\frac{n-3}{4}}+1,n\right]$, two structures for each $a \in [n+1,4n-9] \cup [4n-11,4n-7] \cup {4n-5}$, and one structure for each $a \in \{4n-10,4n-6,4n-3,4n-2,4n-1,4n+1\}$.  In particular, the largest value of $a_1$ in any smooth structure is $8n+2$.  It is interesting to note that the biggest choice of $b_1$ on any smooth arithmetical structure on $\CP_{2,n}$ occurs when $a_1=4n+2$ and $b_1 = 4n^2+4n$.

We can now use the results of Theorem \ref{T:smoothP2n} to count the total number of arithmetical structures on $\CP_{2,n}$:

\begin{theorem}\label{T:P2n}
The number of arithmetical structures on $\CP_{2,n}$ is given by:
\[|Arith(\CP_{2,n})| = 9C_{n+1}-16C_n+5C_{n-1}-C_{n-2}+ \sum_{j=2}^{\floor{\frac{n+1}{4}}} C(n+1,4j)\]

In particular, for large $n$ we compute that

\[\lim_{n \to \infty} \frac{|\Arith(\CP_{2,n})|}{C_{n+1}} =\frac{76523}{14400} \approx 5.3141\]
\end{theorem}

\begin{proof}
We recall from the discussion leading to Corollary \ref{C:count} that any arithmetical structure on $\CP_{2,n}$ can be obtained by taking a smooth structure on a small graph and subdividing it appropriately.  In particular, we can use the results from Theorems \ref{T:smoothP1n} and \ref{T:smoothP2n} and the properties of the function $C(n,k)$ established in Lemma \ref{L:C} to compute:

%
%

\begin{eqnarray*}
|Arith(\CP_{2,n})| &=& \sum_{n'=1}^n \left(C(2,1) \cdot C(n,n') \cdot |\SArith(\CP_{1,n})| + C(2,2) \cdot C(n,n') \cdot |\SArith(\CP_{2,n})|\right) \\
&=&\sum_{n'=1}^n 2 C(n,n') \left(|\SArith(\CP_{1,n})| + \SArith(\CP_{1,n})|\right)\\
&=& 6C(n,1)+11C(n,2)+21C(n,3) + \sum_{n'=4}^n \left(9n'-7+\floor*{\frac{n'-3}{4}}\right)C(n,n')\\
&=& \sum_{n'=1}^n(9n'-7)C(n,n') +C(n,3)+4C(n,1) + \sum_{n'=7}^n \floor*{\frac{n'-3}{4}}C(n,n')\\
&=& 9\sum_{n'=1}^n n'C(n,n') -7  \sum_{n'=1}^n C(n,n')+C(n,3)+4C(n,1) + \sum_{j=2}^{\floor{\frac{n+1}{4}}} \sum_{k=4j-1}^n C(n,k)\\
&=& 9 (C_{n+1}-C_n) - 7 C_n +(C_{n-1}-C_{n-2})+4C_{n-1} + \sum_{j=2}^{\floor{\frac{n+1}{4}}} C(n+1,4j)\\
&=& 9C_{n+1}-16C_n+5C_{n-1}-C_{n-2}+ \sum_{j=2}^{\floor{\frac{n+1}{4}}} C(n+1,4j)\\
\end{eqnarray*}

This proves the formula in the first statement of the theorem.  For large $k$, we again note that $C_{k-1} \approx \frac{C_k}{4}$.  Moreover, it follows from Lemma \ref{L:C} that $C(n+1,4j) \approx \frac{4j}{2^{4j+1}}C_{n+1} \approx \frac{j}{2^{4j-3}}C_n$.  Therefore, the total number of arithmetical structures on $\CP_{2,n}$ approaches \[(9\cdot 4-16+\frac{5}{4}-\frac{1}{4^2}+\sum_{j=2}^\infty \frac{j}{2^{4j-3}})C_{n} =\frac{76523}{57600} C_{n}\]

\end{proof}
Table \ref{T:P2n} gives the number of smooth arithmetical structures as well as the overall number of arithmetical structures on $\CP_{2,n}$ for each $1 \le n \le 10$.

\begin{table}
\centering
$\begin{array}{|c|c|c|}
  \hline
  n & |\SArith(\CP_{2,n})|& |\Arith(\CP_{2,n})|  \\
  \hline
  1 & 3 & 6 \\
  2 & 8 & 17 \\
  3 & 17 & 55 \\
  4 & 25 & 177 \\
  5 & 34 & 581 \\
  6 & 43 & 1945 \\
  7 & 53 & 6625 \\
  8 & 62 & 22899 \\
  9 & 71 & 80137 \\
  10& 80 & 283426 \\
  11 &90 & 1011561\\
  12& 99 & 3638862\\
  13&108&13180428\\
  14&117&48031613\\
  15&127&175978875\\
  \hline
\end{array}$
\caption{The number of smooth structures and overall arithmetical structures on $\CP_{2,n}$ \label{T:P2n}}
\end{table}

\section{Graphs of the form $\CP_{3,n}$}\label{S:P3n}

As in previous sections, we begin by first counting the number of smooth structures on $\CP_{3,n}$.  Similar to the last section, we note that it is easier to relax the $gcd$ condition on the $\rr$-vector and instead consider the $\r$-vector for each structure, which is defined as having $a_3=2, a_2=2t,a_1=2a,b_1=b$.  Moreover, we note that $t|a+1$, so $a=\ell t -1$ for some $\ell \ge 2$ and $a|b+t$ so we have that $b=ka-t$ for some $k>1$.  In fact, we will get a smooth structure on $\CP_{3,n}$ precisely for an triple of integers $(t,k,\ell)$ so that $t,\ell \ge 2$, $k \ge 1$, and $n+1=F(4a,b)=F(4\ell t - 4,k \ell t - k - t)$.  In order to count the number of such triples, we will have to break into different cases based on the values of $t,k,l \pmod 4$.  We will explicitly work through a couple of these cases, and we summarize the full results in Table \ref{T:P3ncases}.

\begin{table}[!htbp]
\centering
$\begin{array}{|c|c|c|c|c|}
  \hline
  k \pmod 4 & \ell \pmod 4 & t \pmod 4 & \text{Equation} & \text{Number} \\
  \hline
  \hline
  1 & 1 & \text{all} & n+9 =\frac{k-1}{4} + \frac{\ell-5}{4} + 4(t-2) & \p'(n-7) \\
  \hline
  1 & 2 & 1 & n-1=\frac{k-1}{4}+ \frac{\ell-2}{4} + \frac{t-5}{4} & \p(n-1) \\
  \hline
  1 & 2 & 2 & n-1=\frac{k-1}{4}+ \frac{\ell-2}{4} + \frac{t-2}{4} & \p(n-1) \\
  \hline
  1 & 2 & 3 & n-1=\frac{k-1}{4}+ \frac{\ell-2}{4} + \frac{t-3}{4} & \p(n-1) \\
  \hline
  1 & 2 & 0 & n-3=\frac{k-1}{4}+ \frac{\ell-2}{4} + \frac{t-4}{4}& \p(n-3)\\
  \hline
  1 & 3 & \text{all} & n-2 = \frac{k-1}{4}+ \frac{\ell-3}{4} + (t-2) & \p(n-2) \\
  \hline
  1 & 0 & 1 & n-3=\frac{k-1}{4}+ \frac{\ell-4}{4} + \frac{t-5}{4} & \p(n-3) \\
  \hline
  1 & 0 & 2 & n-5 = \frac{k-1}{4}+ \frac{\ell-4}{4} + \frac{t-6}{4} & \p(n-5)\\
  &&&\text{ or } t=2, \, \frac{k-1}{4}+ \frac{\ell-4}{4} = n-3 &  \pp(n-3)\\
  \hline
  1 & 0 & 3 & n-2 =  \frac{k-1}{4}+ \frac{\ell-4}{4} + \frac{t-3}{4}  & \p(n-2) \\
  \hline
  1 & 0 & 0 & n-3 = \frac{k-1}{4}+ \frac{\ell-4}{4} + \frac{t-4}{4} & \p(n-3) \\
  \hline
  2 & \text{all} & 1 & n-4 = \frac{k-2}{4} + (\ell-2) + \frac{t-5}{4} & \p(n-4) \\
  \hline
  2 & \text{all} & 2 & n-1 = \frac{k-2}{4} + (\ell-2) + \frac{t-2}{4} & \p(n-1) \\
  \hline
  2 & \text{all} & 3 & n-2 = \frac{k-2}{4} + (\ell-2) + \frac{t-3}{4} & \p(n-2) \\
  \hline
  2 & \text{all} & 0 & n-2 = \frac{k-2}{4} + (\ell-2) + \frac{t-4}{4} & \p(n-2) \\
  \hline

 3 & 1 & \text{all} & n-4 = \frac{k-3}{4}+ \frac{\ell-5}{4} + (t-2) & \p(n-4) \\
 \hline
 3 & 2 & 1 & n-5 = \frac{k-3}{4} + \frac{\ell-6}{4} + \frac{t-5}{4} & \p(n-5) \\
 &&&\text{ or } \ell=2, \frac{k-3}{4} + \frac{t-5}{4} = n-3& \pp(n-3) \\
 \hline
 3 & 2 & 2 & n-7 = \frac{k-3}{4} + \frac{\ell-6}{4} + \frac{t-6}{4} & \p(n-7) \\
 &&& \text{ or } t=2, \, \frac{k-3}{4}+ \frac{\ell-6}{4} = n-5 & \pp(n-5) \\

& & &\text{ or } \ell=2, \, \frac{k-3}{4} + \frac{t-6}{4} = n-5 & \pp(n-5) \\
&&& \text{ or } k=3, \, \ell=2, \, \frac{k-3}{4} = n-3 & 1 \text{ if }n \ge 3 \\
 \hline
 3 & 2 & 3 & n-4 = \frac{k-3}{4} + \frac{\ell-6}{4} + \frac{t-3}{4} & \p(n-4) \\
 &&&\text{ or } \ell=2, \,\frac{k-3}{4} + \frac{t-3}{4} = n-2&  \pp(n-2) \\
 \hline
 3 & 2 & 0 & n-5 = \frac{k-3}{4} + \frac{\ell-6}{4} + \frac{t-4}{4} & \p(n-5) \\
 &&& \frac{k-3}{4} + \frac{t-4}{4} = n-3& \pp(n-3) \\
 \hline
 3 & 3 & \text{all} & n-9 = \frac{k-3}{4}+ \frac{\ell-7}{4} + 4(t-2) & \p'(n-9) \\
 &&& \text{ or } \ell=3, \, n-7 = \frac{k-3}{4} + 4(t-2) &  \pp'(n-7) \\
 \hline
3 & 0 & 1 & n-3 = \frac{k-3}{4}+ \frac{\ell-4}{4} + \frac{t-5}{4} & \p(n-3) \\
  \hline
  3 & 0 & 2 & n-3 = \frac{k-3}{4}+ \frac{\ell-4}{4} + \frac{t-2}{4} & \p(n-3) \\
  \hline
  3 & 0 & 3 & n-3 = \frac{k-3}{4}+ \frac{\ell-4}{4} + \frac{t-3}{4} & \p(n-3) \\
  \hline
   3 & 0 & 1 & n-5 = \frac{k-3}{4}+ \frac{\ell-4}{4} + \frac{t-4}{4} & \p(n-5) \\
 \hline
0 & \text{all} & 1 & n-7 = \frac{k-4}{4}+4(\ell-2) + \frac{t-5}{4} & \p'(n-7) \\
 \hline
0 & \text{all} & 2 & n-6 = \frac{k-4}{4}+4(\ell-2) + \frac{t-2}{4} & \p'(n-6) \\
 \hline
0 & \text{all} & 1 & n-8 = \frac{k-4}{4}+4(\ell-2) + \frac{t-3}{4} & \p'(n-8) \\
 \hline
0 & \text{all} & 1 & n-4 = \frac{k-4}{4}+4(\ell-2) + \frac{t-4}{4} & \p'(n-6) \\
\hline
\end{array}$
\caption{\text{The number of smooth structures on $\CP_{3,n}$ in various cases}\label{T:P3ncases}}
\end{table}

If $k \equiv \ell \equiv 1 \pmod 4$ then we use Lemma \ref{L:props} to compute that
\begin{eqnarray*}
n+1&=&F(4a,b)\\
n+1&=&F(4\ell t - 4,k \ell t - k - t)\\
n+1&=&\frac{k-1}{4} + F(4 \ell t - 4,\ell t -t - 1) \\
n+1&=&\frac{k-1}{4} + F(4t,(\ell-1)t-1) \\
n+1&=&\frac{k-1}{4}+\frac{\ell-5}{4} +F(4t,4t-1) \\
n+1&=&\frac{k-1}{4}+\frac{\ell-5}{4}+4t \\
n-7 &=&\frac{k-1}{4} + \frac{\ell-5}{4} + 4(t-2) \\
\end{eqnarray*}

In particular, we have shown that the smooth structures in this category are in bijection with ordered triples of nonnegative integers $(x,y,z)$ so that $x+y+z=n+9$ and $4|z$.

We next consider the case where $k \equiv 1 \pmod 4$ and $\ell \equiv 2 \pmod 4$.  As in the previous case, we can compute that
\[n+1=F(4a,b)=\frac{k-1}{4}+ F(4t,(\ell-1)t-1) = \frac{k-1}{4}+ \frac{\ell-2}{4} + F(4t,t-1) = \frac{k-1}{4}+ \frac{\ell-2}{4} + F(4,t-1)\]
Recalling Lemma \ref{L:4}, we now have that $n = \frac{k-1}{4}+ \frac{\ell-2}{4} + \frac{t-1+\epsilon_{t-1}}{4}-1$.  For example, if $t \equiv 1 \pmod 4$ we have that
\[n-1=\frac{k-1}{4}+ \frac{\ell-2}{4} + \frac{t-5}{4}\] which shows us that the smooth structures in this category are in bijection with ordered triples of nonnegative integers $(x,y,z)$ so that $x+y+z=n-1$.    More generally, we make the following notational definition:

\begin{definition}
For any integer $n$, we define $\p(n)$ to be the number of ordered triples of nonnegative integers $(x,y,z)$ so that $x+y+z=n$ and $\pp(n)$ to be the number of ordered pairs of nonnegative integers $(x,y)$ so that $x+y=n$.  We further define $\p'(n)$ (resp. $\pp'(n)$) to be the number of triples (resp. pairs) of nonnegative integers summing to $n$ with the further restriction that $4|x$.
\end{definition}

We note that it is well established in the literature (for example, \cite{Stanley}) that $\pp(n)=\binom{n+1}{1}$ and $\p(n) = \binom{n+2}{2}$.  It is straightforward to check that if $n \ge 0$ then $\pp'(n)=\floor*{\frac{n+4}{4}}$.  The function $\p'(n)$ is discussed in \cite[A130519]{OEIS}, and we will use the following lemma about it:

\begin{lemma}\label{L:recur}
For $n \ge 0$ we have that $\p'(n+4)=\p'(n)+n+5$.
\end{lemma}

\begin{proof}
We note that if $(x,y,z)$ is a triple so that $x+y+z=n$ and $4|x$ then by setting $x'=x+4$ we will obtain a new triple $(x',y,z)$ whose entries sum to $n+4$ so that $4|x'$.  Moreover, all such triples are attained in this way except for the $n+5$ triples of the form $(0,y,n+4-y)$ where $0 \le y \le n+4$.  In particular, we see that $\p'(n+4)=\p'(n)+n+5$.
\end{proof}

\begin{theorem}\label{T:smoothP3n}
For all $n \ge 5$ we have that the number of smooth arithmetical structures on $\CP_{3,n}$ is given by:
\[|\SArith(\CP_{3,n})| =\begin{cases}
\frac{47}{4}n^2-36n+\frac{205}{4}, & n \equiv 1 \pmod 4 \\
\frac{47}{4}n^2-36 n+51, & n \equiv 2 \pmod 4 \\
\frac{47}{4}n^2-36 n+\frac{209}{4}, & n \equiv 3 \pmod 4\\
\frac{47}{4}n^2-36n+52, & n \equiv 0 \pmod 4 \\
\end{cases}\]

Moreover, $|\Arith(\CP_{3,1})| = 4$, $|\Arith(\CP_{3,2})| = 17$, $|\Arith(\CP_{3,n})| = 48$, and $|\Arith(\CP_{3,4})|=95$.

\end{theorem}

\begin{proof}
The results of Table \ref{T:P3ncases} can be summarized by noting that for $n \ge 3$ the number of smooth structures on $\CP_{3,n}$ is given by
\begin{eqnarray*}
|\SArith( \CP_{3,n})| &=&  4\p(n-1)+4\p(n-2)+6\p(n-3)+3\p(n-4)+4\p(n-5)+ \p(n-7)\\
&& +2\p'(n-6)+2\p'(n-7)+\p'(n-8)+\p'(n-9)\\
&&+\pp(n-2)+3\pp(n-3)+2\pp(n-5)+\pp'(n-7)+1
\end{eqnarray*}

From this, we can compute the values of $|\Arith(\CP_{3,n})|$ directly for $n \le 4$.  As discussed above, if $n \ge 0$ we have that $\pp(n)=n+1$ and $\p(n) = \frac{(n+2)(n+1)}{2}$, which allows us to simplify the above formula to get that for all $n \ge 5$ we have that:

\[|\SArith( \CP_{3,n})| = 11n^2-30n+40+2\p'(n-6)+2\p'(n-7)+\p'(n-8)+\p'(n-9)+\pp'(n-7)\]

For notational convenience, we set $h(n)=2\p'(n-6)+2\p'(n-7)+\p'(n-8)+\p'(n-9)+\pp'(n-7)$.  It is straightforward to compute that we have $h(5)=0, h(6)=2, h(7)=7$ and $h(8)=12$.  Moreover, Lemma \ref{L:recur} gives us the following relationship:
\begin{eqnarray*}
h(n+4)&=&2\p'(n-2)+2\p'(n-3)+\p'(n-4)+\p'(n-4)+\pp'(n-3)\\
&=&2(\p'(n-6)+n-1)+2(\p'(n-7)+n-2)+(\p'(n-8)+n-3)+(\p'(n-9)+n-4)+(\pp'(n-7)+1)\\
&=&2\p'(n-6)+2\p'(n-7)+\p'(n-8)+\p'(n-9)+\pp'(n-7) + 6n-12\\
&=&h(n)+6n-12
\end{eqnarray*}

By recursion, we see that $h(n+4k)=h(n)+12k^2-24k+6nk$.  Combining these with the values computed above, we obtain that

\[h(n)=\begin{cases}
\frac{3n^2}{4} -6 n+ \frac{45}{4} & n \equiv 1 \pmod 4 \\
\frac{3 n^2}{4}-6 n+ 11 & n \equiv 2 \pmod 4\\
\frac{3 n^2}{4}-6 n+\frac{49}{4} & n \equiv 3 \pmod 4 \\
\frac{3n^2}{4} -6n + 12 & n \equiv 0 \pmod 4 \\
\end{cases}\]

The theorem follows.

\end{proof}

\begin{table}[!htbp]
\centering
$\begin{array}{|c|c|c|}
  \hline
  n & |\SArith(\CP_{3,n})|& |\Arith(\CP_{3,n})|  \\
  \hline
  1 &  4& 16 \\
  2 & 17& 55 \\
  3 &  48& 200 \\
  4 &  95&  698\\
  5 & 165&  2433\\
  6 & 258&  8529\\
  7 & 376&  30126\\
  8 & 516 &  107227\\
  9 & 679 & 384414\\
  10& 866 &  1387312\\
  11& 1078 & 5036958\\
  12& 1312 & 18388019\\
  13& 1569&67460437\\
  14& 1850&248605003\\
  15& 2156&919896078\\
  16& 2484&3416474991\\
  17& 2835& 12731777602\\
  18& 3210& 47593535704\\
  19& 3610& 178420933448\\
  20& 4032& 670633847016\\
  \hline
\end{array}$
\caption{The number of smooth structures and overall arithmetical structures on $\CP_{3,n}$ \label{T:P3n}}
\end{table}

We can now use Theorems \ref{T:smoothP3n} and Corollary \ref{C:count} to explicitly count the total number of arithmetical structures on $\CP_{3,n}$.    In order to do this, we will define auxiliary functions $\eta_i(k)$ which give the difference between the actual number of smooth arithmetical structures on $\CP_{i,k}$ and the number predicted by the polynomials in the formulae in Theorems \ref{T:smoothP1n}, \ref{T:smoothP2n}, and \ref{T:smoothP3n}.  Explicitly, we have the following:

\begin{center}
\begin{tabular}{|c||c|c|c|}
   \hline
   $k$ & $\eta_1(k)$ & $\eta_2(k)$ & $\eta_3(k)$ \\
   \hline
   \hline
   1 & -1& 6 & -23 \\
   \hline
   2 & -1 & 2 & -9 \\
   \hline
   3 & 0 & 1 & -2 \\
   \hline
   4 & 0 & 0 & -1 \\
   \hline
   >4 & 0 & 0 & 0 \\
   \hline
 \end{tabular}
\end{center}

We will also define the function $\gamma(k)$ as follows:

\[\gamma(k)=\begin{cases}
\frac{3}{4} & k \equiv 1 \pmod 4 \\
0 & k \equiv 2 \pmod 4 \\
\frac{11}{4} & k \equiv 3 \pmod 4 \\
2 & k \equiv 0 \pmod 4\\
\end{cases}\]

Then it follows from our earlier results that \[|\SArith(\CP_{3,k})| + 2|\SArith(\CP_{2,k})| + 2|\SArith(\CP_{1,k})| = \frac{47}{4}k^2 - \frac{35}{2} k + 34 + \gamma(k) +\eta_3(k)+2\eta_2(k)+2\eta_1(k)\]

In particular, this allows us to compute:

\begin{eqnarray*}
|\Arith(\CP_{3,n})| &=& \sum_{k=1}^n C(n,k)\left(|\SArith(\CP_{3,k})|+2|\SArith(\CP_{2,k})|+2|\SArith(\CP_{1,k})|\right) \\
&=&\sum_{k=1}^n C(n,k) \left(\frac{47}{4}k^2 - \frac{35}{2} k + 34 + \gamma(k) \right)+\sum_{k=1}^n \left(\eta_3(k)+2\eta_2(k)+2\eta_1(k)\right)C(n,k)\\
&=&\sum_{k=1}^n C(n,k) \left(\frac{47}{4}k^2 - \frac{35}{2} k + 34 + \gamma(k) \right)-C(n,4)-7C(n,2)-13C(n,1) \\
&=&\frac{47}{4}(2C_{n+2}-5C_{n+1}+C_n)-\frac{35}{2}(C_{n+1}-C_n)+ 34 C_n -C(n,4) -7C_{n-1}-13C_{n-1} \sum_{k=1}^n \gamma(k)C(n,k) \\
&=&\frac{47}{2}C_{n+2} - \frac{305}{4} C_{n+1} + \frac{253}{4}C_n -(C_{n-1}-2C_{n-2})-20C_{n-1} + \sum_{k=1}^n \gamma(k)C(n,k) \\
&=&\frac{47}{2}C_{n+2} - \frac{305}{4} C_{n+1} + \frac{253}{4}C_n -21C_{n-1}+2C_{n-2} + \frac{3}{4}\sum_{k\equiv 1}^n C(n,k)+ \frac{11}{4}\sum_{k\equiv 3}^n C(n,k) + 2 \sum_{k\equiv 0}^n C(n,k) \\
\end{eqnarray*}

This gives an explicit formula for the number of arithmetical structures on $\CP_{3,n}$.  For large values of $n$, we recall that Lemma \ref{L:C} tells us that $C_i \approx 4C_{i-1}$.  Moreover we have that $C(n,k) \approx \frac{k}{2^{k+1}}$, which can be used to show us that

\begin{alignat*}{2}
\sum_{k\equiv 1}^n C(n,k)&\approx&\sum_{j=0}^\infty \frac{4j+1}{2^{4j+2}}C_n&\approx \frac{76}{225}C_n\\
\sum_{k\equiv 3}^n C(n,k)&\approx&\sum_{j=0}^\infty \frac{4j+3}{2^{4j+4}}C_n&\approx \frac{49}{225}C_n\\
\sum_{k\equiv 0}^n C(n,k)&\approx&\sum_{j=0}^\infty \frac{4j+4}{2^{4j+5}}C_n&\approx \frac{32}{225}C_n\\
\end{alignat*}

In particular, one can see that for large values of $n$ we will have that $|\Arith(\CP_{3,n})| \approx \kappa C_n$, where
\[\kappa = \frac{47}{2} \cdot 4^2 - \frac{305}{4} \cdot 4 + \frac{253}{4}-21 \cdot \frac{1}{4} +2 \cdot \left(\frac{1}{4}\right)^2 + \frac{3}{4} \cdot \frac{76}{225} + \frac{11}{4} \cdot \frac{49}{225} + 2 \cdot \frac{32}{225} = \frac{78157}{600}=130.261\overline{6}\]

%
%
%
%

\bibliographystyle{amsplain}
\bibliography{arithbib}

\end{document}